\documentclass{amsart} \title{Finite Chevalley groups and loop groups}

\author{Masaki Kameko}%\givenname{Masaki}\surname{Kameko}
\address{Department of Mathematics\\
Faculty of Contemporary Society\\
Toyama University of International Studies\\
65-1 Higashikuromaki\\
Toyama, 930-1292\\ Japan}
\email{kameko@tuins.ac.jp}\urladdr{}

\usepackage{pb-diagram}

\newtheorem{theorem}{Theorem}[section]
\newtheorem{proposition}{Proposition}[section]
\newtheorem{lemma}{Lemma}[section]
\newtheorem{conjecture}{Conjecture}[section]

\theoremstyle{definition}

\newtheorem{remark}{Remark}[section]
\makeatletter
\let\c@proposition=\c@theorem
\let\c@lemma=\c@theorem
\let\c@conjecture=\c@theorem
\let\c@definition=\c@theorem
\let\c@remark=\c@theorem
\makeatother

\newcommand{\spin}{\mathrm{Spin}}
\newcommand{\Loop}{\mathcal{L}}
\newcommand{\closure}{\overline{\mathbb{F}}_p}

\newcommand{\HomotopyFibre}[1]{\mathrm{fib}(#1)}
\newcommand{\MappingTrack}[1]{P_{#1}}

\newcommand{\gr}{\mathrm{gr}\thinspace}
\newcommand{\Cohomology}[1]{H^{*}(#1)}
\newcommand{\Homology}[1]{H_{*}(#1)}
\newcommand{\ReducedHomology}[1]{\tilde{H}_{*}(#1)}
\newcommand{\ReducedCohomology}[1]{\tilde{H}^{*}(#1)}

\newcommand{\holim}{\displaystyle \mathop{\mathrm{holim}}_{\leftarrow}}
\newcommand{\spec}[1]{\mathrm{Spec}(#1)}

\begin{document}

\begin{abstract}
Let $p$, $\ell$ be distinct primes  and let $q$ be a power of $p$. Let $G$ be a connected compact Lie group.
We show that there exists an integer $b$ such that the mod $\ell$ cohomology of the classifying space of a finite Chevalley group $G(\mathbb{F}_q)$  is isomorphic to the mod $\ell$ cohomology of the classifying space of the loop group  $\mathcal{L} G$ for $q=p^{ab}$, $a\geq 1$.
%\ \\  \ \\
%{\tt Filename tezuka.tex} \hfill \today
\end{abstract}

\maketitle

\section{Introduction}

Let $p$, $\ell$ be distinct primes  and let $q$ be a power of $p$. We denote by 
$\mathbb{F}_q$ the finite field with $q$-elements. 
Let $G$ be a connected compact Lie group. 
There exists a reductive complex linear algebraic group $G(\mathbb{C})$ associated with $G$, called the complexification of $G$. 
One may consider $G(\mathbb{C})$ as $\mathbb{C}$-rational points of a group scheme  over $\mathbb{C}$ obtained by the base-change of a reductive integral affine group scheme $G_{\mathbb{Z}}$, so-called Chevalley group scheme, with the complex analytic topology.
For a field $k$, taking the $k$-rational points of the group scheme 
$$G_{k}=G_\mathbb{Z}\times_{\spec{\mathbb{Z}}}\spec{k}$$ over $k$, we have the (possibly infinite) Chevalley group $$G(k)=\mathrm{Hom}_{\mathrm{Sch}/k}(\spec{k}, G_{k}),$$
where $\mathrm{Sch}/k$ is the category of schemes over $k$. We consider the Chevalley group $G(k)$ as a discrete group unless otherwise is clear from the context.
Denote by $\closure$ the algebraic closure of the finite field $\mathbb{F}_q$.
We may consider the finite Chevalley group $G(\mathbb{F}_q)$ as the fixed point set $G(\closure)^{\phi^q}$ where
\[
\phi^q:G(\closure)\to G(\closure)
\]
is the Frobenius map induced by the Frobenius homomorphism $\phi^q:\closure\to \closure$ sending $x$ to $x^q$.

In \cite{quillen1972}, Quillen computed the mod $\ell$ cohomology of a finite general linear group $GL_n(\mathbb{F}_q)$. The finite general linear group $GL_n(\mathbb{F}_q)$ is a finite Chevalley group associated with the unitary group $U(n)$.
We recall Quillen's computation from the viewpoint of  the the following Theorem~\ref{Theorem:fm} due to Friedlander \cite[Theorem 12.2]{friedlander1982}, Friedlander-Mislin \cite[Theorem 1.4]{friedlanderMislin1984}.

Throughout the rest of this paper, we fix a connected compact Lie group $G$ and associated reductive integral affine group scheme $G_{\mathbb{Z}}$.
Let $BG^\wedge$ be the Bousfield-Kan $\mathbb{Z}/\ell$-completion of the classifying space $BG$ of the connected compact Lie group $G$.
We write  $\Cohomology{X}$, $\ReducedCohomology{X}$ for the mod $\ell$ cohomology, reduced mod $\ell$ cohomology  of a space $X$, respectively. We also write $\Homology{X}$, $\ReducedHomology{X}$ for the mod $\ell$ homology, reduced mod $\ell$ homology of $X$, respectively.
We denote by $\HomotopyFibre{\alpha}$, $\pi_0:\MappingTrack{\alpha}\to X$ the homotopy fibre, mapping track of a map $\alpha:A\to X$. 
That is, \[
\MappingTrack{\alpha}=\{ (a,\lambda) \in A \times X^{I}\;|\; \alpha(a)=\lambda(1) \},
\]
$
\pi_0((a,\lambda))=\lambda(0)
$
and 
$
\HomotopyFibre{\alpha}=\pi_0^{-1}(*),
$
where $I=[0,1]$ is the unit interval, $X^{I}$ is the set of continuous maps from $I$ to $X$, $*$ is the base-point of $X$.

%%%%%%%%%%%%%%%%%%%%%%%%%%%%%%%%%%%%%%%%%%%%%%%%%%%%%%%%%
%:  Theorem [fm]

\begin{theorem}[Friedlander-Mislin]  \label{Theorem:fm} 
There exist maps 
\[
D:BG(\closure) \to BG^\wedge
\]
and 
\[
\phi^q:BG^\wedge \to BG^\wedge
\]
satisfying the following three conditions{\rm :}\\
{\rm (1)}
The induced homomorphism 
$
D^{*}:\Cohomology{BG^\wedge}\to \Cohomology{BG(\closure)}
$
is an isomorphism.\\
{\rm (2)}
 $\phi^q \circ D = D \circ \phi^q$ where 
$
\phi^q:BG(\closure) \to BG(\closure)
$
is the Frobenius map induced by the Frobenius homomorphism $\phi^q:\closure\to \closure$. \\
{\rm (3)} There exists a map 
$
\HomotopyFibre{D_q} \to \HomotopyFibre{\Delta}
$
induces an isomorphism 
$
\Cohomology{\HomotopyFibre{\Delta}} \to\Cohomology{\HomotopyFibre{D_q}},
$
where the above map is obtained from the following  homotopy commutative diagram by choosing a suitable homotopy{\rm :}
\[
\begin{diagram}
\node{BG(\mathbb{F}_q)} \arrow{e}\arrow{s,l}{D_q} \node{BG^\wedge}\arrow{s,r}{\Delta} \\
\node{BG^\wedge} \arrow{e,t}{(1\times \phi^q) \circ \Delta} \node{BG^\wedge \times BG^\wedge,}
\end{diagram}
\]
where $D_q=D\circ i_q$, $i_q=BG(\mathbb{F}_q)\to BG(\closure)$ is the map induced by the inclusion of $\mathbb{F}_q$ into $\closure$ and 
$\Delta:BG^\wedge \to BG^\wedge \times BG^\wedge$ is the diagonal map.
\end{theorem}

\begin{remark}\label{Remark:strict1}
In \cite[Proposition 8.8]{friedlander1982}, Friedlander constructed a chain of maps between simplicial sets  
$\holim   \thinspace (\mathbb{Z}/\ell)_{\infty}( BG_{\closure})_{\mathrm{et}}$ and $(\mathbb{Z}/\ell)_{\infty} \mathrm{Sing}_{\bullet}(BG(\mathbb{C}))$, 
where $G(\mathbb{C})$ is given the complex analytic topology.
He showed that these maps are weak homotopy equivalences. We take $BG^\wedge$ to be  the geometric realization of the simplicial set $\holim  \thinspace (\mathbb{Z}/\ell)_{\infty}  (BG_{\closure})_{\mathrm{et}}$, so that the Forbenius map $\phi^{q}:BG^\wedge \to BG^\wedge$ is induced by the map defined on $G_{\closure}$. Therefore, the map $\phi^q$ is an automorphism and  there holds \[
\underbrace{\phi^q \circ \dots \circ \phi^q}_{e-\text{times}}=\phi^{q^e}. \]
We emphasize here that the equality holds in the category of sets and maps, not in the  homotopy category.
\end{remark}

\begin{remark} \label{Remark:strict2}
For a discrete group $H$, we may identify  the classifying space $BH$ with  the geometric realization of $\holim  \thinspace (BH_{k})_{\mathrm{et}}$, where $k$ is an algebraically closed field and $H_{k}$ is a group scheme $H \otimes \mathrm{Spec}(k)$.
The map $D$ in Theorem~\ref{Theorem:fm} is induced by the obvious homomorphism of group schemes
\[
G(k)_{k}=\mathrm{Hom}_{Sch/k}(\mathrm{Spec}(k), G_{k}) \otimes \mathrm{Spec}(k) \to G_{k}.
\]
See Friedlander-Mislin \cite[Section 2]{friedlanderMislin1984} for detail.
Thus, we have the equality $\phi^q \circ D=D \circ \phi^q$  in Theorem~\ref{Theorem:fm}. 
\end{remark}

%:  Quillen's computaion

Now, we recall Quillen's computation of the mod $\ell$ cohomology of $GL_{n}(\mathbb{F}_q)$. 
The first part of Quillen's computation is the homotopy theoretical interpretation of the problem.
For a map $f :\thinspace X \to X$, let us define a space $\Loop_f X$  by
\[
\Loop_f X=\{ \lambda \in X^{I} \;|\; \lambda(1)=f(\lambda(0)) \}.
\]
We call this space  the {\it twisted loop space} of $f$ following the terminology of  \cite{kishimoto}. Let $\pi_0:\Loop_{f} X \to X$ be the evaluation map at $0$.
Let $\pi_0:\MappingTrack{\Delta}\to X \times X$ be the mapping track of the diagonal map $\Delta:X \to X \times X$.
Associated with the diagram in Theorem~\ref{Theorem:fm} (3), we have the following fibre square:
\[
\begin{diagram}
\node{\Loop_f X} \arrow{e,t}{g} \arrow{s,r}{\pi_0} \node{\MappingTrack{\Delta}} \arrow{s,r}{\pi_0} \\
\node{X} \arrow{e,t}{(1\times f) \circ \Delta} \node{X\times X,}
\end{diagram}
\]
where  $g$ is given by
\[
g(\lambda)=(\lambda(0), \lambda')\in  X \times (X\times X)^I,
\]
and $\displaystyle \lambda'(t)=(\lambda(\frac{t}{2}), \lambda(1-\frac{t}{2}))$.
Theorem~\ref{Theorem:fm} (3) implies that  $\Cohomology{\Loop_f X}$ is isomorphic to $\Cohomology{BG(\mathbb{F}_q)}$ for $X=BG^\wedge$, $f=\phi^q$.
Thus, the computation of the mod $\ell$ cohomology of a finite Chevalley group is nothing but the computation of the mod $\ell$ cohomology of  the twisted loop space $\Loop_f X$.  

The second part of Quillen's computation is the computation using the Eilenberg-Moore spectral sequence.
For a twisted loop space,  there exists the Eilenberg-Moore spectral sequence converging to the associated graded algebra of the mod $\ell$ cohomology of the twisted loop space  $\Loop_f X.$
Let us write $A$ for $\Cohomology{X}$. 
The  $E_2$-term of the Eilenberg-Moore spectral sequence is given by 
$
\mathrm{Tor}_{A \otimes A}(A, A).
$
 If   the induced homomorphism 
$
{f}^*:A \to A
$
is the identity homomorphism and if $A$ is a polynomial algebra, then the above $E_2$-term is a polynomial tensor exterior algebra  $A \otimes V$ where $V=\mathrm{Tor}_{A}(\mathbb{Z}/\ell, \mathbb{Z}/\ell)$, 
and since as an algebra over $\mathbb{Z}/\ell$, it is generated by $\mathrm{Tor}^0_{A\otimes A}(A, A)$ 
and $\mathrm{Tor}^{-1}_{A\otimes A}(A,A)$, the spectral sequences collapses at the $E_2$-level.

%:  Free loop space and loop group

On the other hand, it is well-known that there exists a homotopy equivalence between the classifying space of the loop group $\Loop G$ and the free loop space $\Loop BG$, where
\[
\Loop X=\{\lambda\in X^{I} \;|\; \lambda(1)=\lambda(0) \}.
\]
See Proposition 2.4 in \cite{atiyahBott}.
For the free loop space $\Loop BG$, 
we have the following fibre square:
\[
\begin{diagram}
\node{\Loop X} \arrow{e} \arrow{s,l}{\pi_0} \node{\MappingTrack{\Delta}}\arrow{s,r}{\pi_0} \\
\node{X} \arrow{e,t}{\Delta} \node{X\times X,}
\end{diagram}
\]
where $\pi_0$ is the evaluation map at $0$, so that $\pi_0(\lambda)=\lambda(0)$, and $\pi_0:\MappingTrack{\Delta} \to X \times X$ is the mapping track of the diagonal map $\Delta:X \to X\times X$.
As in the case of finite Chevalley groups, there exists the Eilenberg-Moore spectral sequence 
\[
\mathrm{Tor}_{A\otimes A} (A, A) \Rightarrow \gr   \Cohomology{\Loop BG}.
\]
Thus, it is easy to see that if $A=\Cohomology{BG}$ is a polynomial algebra, the $E_2$-term of the spectral sequence is equal to the polynomial tensor exterior algebra $A \otimes V=A \otimes \mathrm{Tor}_{A}(\mathbb{Z}/\ell,\mathbb{Z}/\ell)$, the spectral sequence collapses at the $E_2$-level as in the case of finite Chevalley groups.
Therefore, if  $\Cohomology{BG}$ is a polynomial algebra, 
the mod $\ell$ cohomology of the free loop space of the classifying space $BG$ is isomorphic to the mod $\ell$ cohomology of the finite Chevalley group $G(\mathbb{F}_q)$ as  a graded $\mathbb{Z}/\ell$-module.

Even if $\Cohomology{BG}$ is not a polynomial algebra over $\mathbb{Z}/\ell$,  if the induced homomorphism ${\phi^q}^*:A\to A$ is the identity homomorphism, $E_2$-terms of the above Eilenberg-Moore spectral sequences are the same.
Observing this phenomenon, Tezuka in \cite{tezuka1999} asked  the following:
\begin{conjecture}\label{t}
If $\ell | q-1$ (resp. $4 | q-1$) when $\ell$ is odd (resp. even), 
there exists a ring isomorphism 
between $\Cohomology{BG(\mathbb{F}_q)}$ and $\Cohomology{\Loop BG}$.
\end{conjecture}

%:  Theorem [main]

In conjunction with this conjecture, in this paper, we prove  the following result:

\begin{theorem}\label{Theorem:main}
There exists an integer $b$ such that, for $q=p^{ab}$ where $a$ is an arbitrary positive integer,  there exists an
 isomorphism of graded $\mathbb{Z}/\ell$-modules
\[
\Cohomology{BG(\mathbb{F}_q)} =\Cohomology{\Loop BG}.
\]
\end{theorem}

\begin{remark}
Although we give an example of the integer $b$ in Theorem~\ref{Theorem:main} in Section~\ref{sec2} as a function of the graded $\mathbb{Z}/\ell$-module
$\Cohomology{G}$, it is not at all the best possible.
\end{remark}

\begin{remark}
If $H^{*}(BG)$ is not a polynomial algebra, it is not easy to compute the mod $\ell$ cohomology of $BG(\mathbb{F}_q)$ and $\Loop BG$. The only computational results in the literature  are the computation of the mod 2 cohomology of $B\spin_{10}(\mathbb{F}_q)$ and $\Loop B \spin(10)$ in \cite{kleinerman1982} and \cite{kuribayashiMimuraNishimoto2006} for $\ell=2$ and 
the mod $3$ cohomology of $\Loop BPU(3)$ for $\ell=3$ in \cite{kuribayashiMimuraNishimoto2006}.
\end{remark}

When we want to  show that the cohomology of a space $X$ is isomorphic to the cohomology of another space $Y$, we usually try to construct a chain of maps 
\[
X=X_0 \stackrel{f_0}{\longleftarrow} X_1 \stackrel{f_1}{\longrightarrow} X_2 \longleftarrow \dots \longleftarrow X_n \stackrel{f_{n}}{\longrightarrow} X_{n+1} =Y
\]
such that maps $f_k$'s induce isomorphisms in mod $\ell$ cohomology.
For example, Theorem~\ref{Theorem:fm} is proved by this method. 
However, when we try to prove Theorem~\ref{Theorem:main} or Conjecture~\ref{t}, we can not construct such a chain of maps.
Consider the case $G=S^1$.
Then, we have $G(\mathbb{F}_q)=\mathbb{Z}/(q-1)$, $G(\mathbb{C})=\mathbb{C} \backslash \{0\}$.
So, we have $H^{*}(BG(\mathbb{F}_q);\mathbb{Q})=\mathbb{Q}$, $H^{*}(\Loop BG;\mathbb{Q})=\mathbb{Q}[y] \otimes \Lambda(x)$.
If there exists a chain of maps such as above, then they also induce isomorphisms of Bockstein spectral sequences. This contradicts  the above observation on the rational (and integral) cohomology
of  $BG(\mathbb{F}_q)$ and $\Loop BG$.

Thus, in the proof of Theorem~\ref{Theorem:main}, we construct maps which induce monomorphisms among Leray-Serre spectral sequences associated with fibrations
$\pi_0:\Loop_f X \to X$, $\pi_0:\Loop X \to X$, $\pi_0:\Loop_{f} X \times_{X} \MappingTrack{\alpha} \to X$, where $X=BG^\wedge$, $f=\phi^q$ and $\alpha:A \to X$ is a  certain map we define in Section~\ref{sec3}. 
By 
comparing the image of Leray-Serre spectral sequences, we construct an isomorphism between 
Leray-Serre spectral sequences associated with fibrations $\pi_0:\Loop_f X \to X$ and  $\pi_0:\Loop X \to X$. This isomorphism could not be realized by a chain of maps. 
We announced and outlined the proof of Theorem~\ref{Theorem:main} in \cite{kameko}. By choosing $\phi^q$ and $D$ as in Remarks~\ref{Remark:strict1}, \ref{Remark:strict2}, in this paper, we can give a simpler proof for Theorem~ \ref{Theorem:main}.

In Section~\ref{sec2}, we define the integer $b$ as a function of a graded $\mathbb{Z}/\ell$-module $\Cohomology{G}$.
In Section~\ref{sec3}, we give a proof of Theorem~\ref{Theorem:main} assuming Lemma~\ref{Lemma:mono}.
In Section~\ref{sec4}, we prove Lemma~\ref{Lemma:mono}.

%\ \\ \noindent {\bf Acknowledgements.} 
Since there exists no map realizing the 
isomorphism between $\Cohomology{BG(\mathbb{F}_q)}$ and $\Cohomology{\Loop BG}$, it is difficult to believe the existence of such isomorphism for arbitrary connected compact Lie group $G$. 
It is my pleasure to thank M. Tezuka for informing me of Conjecture~\ref{t}.
The author is partially supported by Japan Society for the Promotion of Science, Grant-in-Aid for Scientific Research (C) 19540105.

%%% section 2 %%%%%%%%%%%%%%%%%%%%%%%%%%%%%%%%%%%%%

\section{The integer $b$} \label{sec2}

In this section, we define the integer $b$ in Theorem~\ref{Theorem:main}.
We define the integer $b$ as 
\[
b={e_1}^{\dim G} {e_2}
\]
and we define $e_1$, $e_2$ as follows:
For the sake of notational simplicity, let  $V=\Cohomology{G}$. We have isomorphisms 
\[
V=H^{*}_{\mathrm{et}}(G_{\closure}, \mathbb{Z}/\ell)=\Cohomology{\HomotopyFibre{D_q}}=\Cohomology{\HomotopyFibre{\Delta}}=\Cohomology{\Omega BG^\wedge}.
\]
Denote by $GL(V)$ be the group of automorphisms of $V$ and we also denote by $|GL(V)|$ the order of the finite group $GL(V)$. 
Let 
\begin{align*}
e_1 & =(\ell |GL(V)|)^{2\dim G}\quad \text{and}\\
e_2 &=|GL(V\otimes V)|.
\end{align*}

Before we proceed to lemmas, we set up some notations. Let us consider a commutative diagram.
\[
\begin{diagram}
\node{A} \arrow{e,t}{g} \arrow{s,l}{\alpha} \node{B} \arrow{s,r}{\beta} \\
\node{X} \arrow{e,t}{f} \node{Y.}
\end{diagram}
\]

We write $f:X^I \to Y^I$ for the map induced by $f:X\to Y$, so that 
\[
f(\lambda)(t)=f(\lambda(t)).
\]
We also write $g\times f$ for the map $A \times X^I \to B\times Y^I$ induced by $f$, $g$ and the restriction of $g\times f:\thinspace A \times X^I \to B \times Y^I$ to the mapping tracks $\MappingTrack{\alpha} \to \MappingTrack{\beta}$ and the homotopy fibres $\HomotopyFibre{\alpha} \to \HomotopyFibre{\beta}$.

Let $q'=q^e$ and $q$ is a power of $p$.
The inclusion of $\mathbb{F}_q$ into $\mathbb{F}_{q'}$ induces  maps
\[
i:\thinspace BG(\mathbb{F}_q) \to BG(\mathbb{F}_{q'})
\]
and
\[
i \times 1:\thinspace \HomotopyFibre{D_q} \longrightarrow \HomotopyFibre{D_{q'}}.
\]

%:  Lemma 2.1 [zero]
%%%%%%%%%%%%%%%%%%%%%%%%%%%%%%%%%%%%%%%%%%%%%%%%%

As for $e_1$, we have the following lemma. A variant of this lemma is used in the proof of Theorem 1.4 in \cite{friedlanderMislin1984}. 

\begin{lemma} \label{Lemma:zero}
Suppose that $e$ is divisible by $e_1$. 
Then, the induced homomorphism
\[
(i\times 1)^*:\ReducedCohomology{\HomotopyFibre{D_{q'}}} \to \ReducedCohomology{\HomotopyFibre{D_q}}
\]
is zero.
\end{lemma}

\begin{proof}
In general, 
we have 
\[
\Delta_{*}({(i \times 1)}_*(x))=1 \otimes {(i \times 1)}_*(x)+ \sum {(i \times 1)}_*( y' )\otimes {(i\times 1)}_{*} (y'' )+ {(i \times 1)}_*(x) \otimes 1,
\]
where $\deg y' <\deg x$ or $\deg y'' <\deg x$.
Hence, if ${(i \times 1)}_*(y)=0$ for $\deg y < \deg x $, then  we have 
\[
\Delta_{*}({(i \times 1)}_*(x))=1 \otimes {(i \times 1)}_*(x)+{(i \times 1)}_*(x) \otimes 1.
\]
So, if $(i\times 1)_{*}(y)=0$ for $\deg y<\deg x$, ${(i \times 1)}_*(x)$ is primitive.

The Frobenius map $\phi^q$ is an element of the Galois group 
$\mathrm{Gal}(\closure/\mathbb{F}_p)$, the induced homomorphism  
${\phi^q}^*:H_{\mathrm{et}}^{*}(G_{\closure}, \mathbb{Z}/\ell) \to H^{*}_{\mathrm{et}}(G_{\closure}, \mathbb{Z}/\ell)$
 is an automorphism in $GL(V)$.
Suppose that  $q'=q^e$ and $e$ is divisible by $\ell \cdot m$ where $m$ is the order of 
\[
{\phi^q}^*:V \to V
\]
as an element in $GL(V)$. Consider the induced homomorphism
\[
{(i \times 1)}_*:\thinspace \ReducedHomology{\HomotopyFibre{D_q}} \to \ReducedHomology{\HomotopyFibre{D_{q'}}}.
\]
The isomorphism between $H_{\mathrm{et}}^{*}(G_{\closure}, \mathbb{Z}/\ell)$ and $\Cohomology{\HomotopyFibre{D_q}}$ is given by 
the Lang map  $(1/\phi^q):  G({\closure})/G(\mathbb{F}_q) \to G({\closure})$ defined by $(1/\phi^q) (g)= g \cdot (\phi^q(g))^{-1}$.
Thus, the map ${(i \times 1)}_*$ corresponds to a homomorphism $\theta^{q'/q}:G(\closure) \to G(\closure)$
given by the diagram 
\[
\begin{diagram}
\node{G({\closure})/G(\mathbb{F}_q)} 
\arrow{e,t}{\pi} 
\arrow{s,l}{1/\phi^q} 
\node{G({\closure})/G(\mathbb{F}_{q'})} 
\arrow{s,r}{1/\phi^{q'}}\\
\node{G({\closure})} 
\arrow{e,t}{\theta^{q'/q}} \node{G({\closure}),}
\end{diagram}
\]
where $\pi$ is the obvious projection.
In other words, $\theta^{q'/q}$ is given by 
\[
\theta^{q'/q}
(g)= g \cdot \phi^q(g) \cdot \dots \cdot (\phi^{q})^{e-1}(g).
\]
Thus,  ${(i \times 1)}_*(x)$ is given by 
\[
{(i \times 1)}_*(x)=(\mu \circ (1 \times \phi^q \times \dots \times \phi^{q^{e-1}})\circ \Delta )_*(x)
\]
If $x$ is primitive, we have 
\begin{align*}
{(i \times 1)}_*(x)&=\sum_{t=0}^{e-1} (\phi^q)_{*}^t (x)\\
&=(e/m) \cdot \left( \sum_{t=0}^{m-1} (\phi^q)_{*}^t(x) \right) \\
&=0.
\end{align*}

Thus, if $e$ is divisible by $(\ell \cdot m)^2$ and  if ${(i\times 1)}_*(y)=0$ for $\deg y<j$, then ${(i \times 1)}_*(x)=0$ for $\deg x \leq j$.
Therefore, if $G$ is connected and if $e$ is divisible by $ (\ell \cdot m)^{2k}$, the induced homomorphism
\[
\ReducedCohomology{\HomotopyFibre{D_{q'}}} \to \ReducedCohomology{\HomotopyFibre{D_q}}
\]
is zero up to degree $k$.  Let $k=\dim G$. Since, by definition,  $H^{j}(\HomotopyFibre{D_q})=\{0\}$ for $j>k$, we have Lemma~\ref{Lemma:zero}.
\end{proof}

%:  Lemma 2.2 [identity]
%%%%%%%%%%%%%%%%%%%%%%%%%%%%%%%%%%%%%%%%%%%%%%%%%%

As for the integer $e_2$, we prove the following:

\begin{lemma}
\label{Lemma:identity}
Suppose that  $e$ is divisible by $e_2$.
Then, the induced homomorphism 
\[
(1 \times (1\times \phi^{q'}))^*:\Cohomology{\HomotopyFibre{\Delta \circ D_{q}}} \to \Cohomology{\HomotopyFibre{\Delta\circ D_{q}}}
\]
is the identity homomorphism.
\end{lemma}

\begin{proof}

Let $p_i:BG^\wedge\times BG^\wedge \to BG^\wedge$ be the projections onto the first and second factors for $i=1,2$.
Consider the diagram 
\[
\begin{diagram}
\node{\Omega BG^\wedge} \arrow{e,t}{=} \arrow{s} \node{\Omega BG^\wedge} \arrow{e,t}{\Omega \Delta} \arrow{s} \node{\Omega(BG^\wedge \times BG^\wedge)} \arrow{s}\\
\node{\HomotopyFibre{\Delta\circ D_q}} \arrow{e,t}{1\times {p_1}} \arrow{s,l}{1\times {p_2}}
\node{\HomotopyFibre{D_q}} \arrow{e,t}{D_q\times \Delta} \arrow{s,l}{p_1} \node{\HomotopyFibre{\Delta}}\arrow{s,r}{p_1}\\
\node{\HomotopyFibre{D_q}} \arrow{e,t}{p_1} \node{BG(\mathbb{F}_q)} \arrow{e,t}{D_q} \node{BG^\wedge}
\end{diagram}
\]
The induced homomorphism 
\[
(\Omega \Delta)_*:\thinspace \Homology{\Omega BG^\wedge} {\longrightarrow} 
\Homology{\Omega (BG^\wedge \times BG^\wedge)}
\]
is a monomorphism and $\pi_1(BG^\wedge)=\{0\}$. So, $\pi_1(BG(\mathbb{F}_q))$ acts trivially on 
$\Homology{\Omega BG^\wedge}$. Therefore, $\pi_1(\HomotopyFibre{D_q})$ also acts trivially on $\Homology{\Omega BG^{\wedge}}$. Thus, the local coefficient of the induced fibre sequence
\[
\Omega BG^\wedge \to \HomotopyFibre{\Delta\circ D_q} \stackrel{1 \times {p_2}}{\longrightarrow}  \HomotopyFibre{D_q}
\]
is trivial.
Hence, the $E_2$-term of the Leray-Serre spectral sequence for the cohomology of $\HomotopyFibre{\Delta\circ D_q}$ is 
given by 
\[
V \otimes V = \Cohomology{\HomotopyFibre{D_q}} \otimes \Cohomology{\Omega BG^\wedge}.
\]
Therefore, we have that 
\[
\dim_{\mathbb{Z}/\ell} \Cohomology{\HomotopyFibre{\Delta\circ D_q}} \leq \dim_{\mathbb{Z}/\ell} (V \otimes V).
\]
Since the  map $\phi^q:BG^\wedge \to BG^\wedge$ is an automorphism, 
the induced map  $$1\times (1\times \phi^q):\HomotopyFibre{\Delta\circ D_q} \to \HomotopyFibre{\Delta \circ D_q}$$ is also an automorphism. Since
\[
\dim_{\mathbb{Z}/\ell} \Cohomology{\HomotopyFibre{\Delta\circ D_q}} \leq \dim_{\mathbb{Z}/\ell} (V \otimes V)
,
\]
 $e_2$ is divisible by the order of 
 $$(1 \times (1\times \phi^q))^*:\Cohomology{\HomotopyFibre{\Delta\circ D_q}} \to 
 \Cohomology{\HomotopyFibre{\Delta \circ D_q}}. $$ Hence, if $e$ is divisible by $e_2$, we have 
\[
(1 \times (1\times \phi^{q'}))^*=((1\times (1\times \phi^q))^*)^e=1.\qedhere
\]
\end{proof}

%%%%%%%%%%%%%%%%%%%%%%%%%%%%%%%%%%%%%%%%%%%%%%

%%% section 3 %%%%%%%%%%%%%%%%%%%%%%%%%%%%%%%%%%%%

\section{Proof of Theorem~\ref{Theorem:main}}\label{sec3}

Let $X$ be a space and let $f:X \to X$ be a self-map of $X$ with a non-empty fixed point set. 
Let $\alpha:A\to X$ be a map such that  
\[
f\circ \alpha = \alpha.
\]
We choose a base-point $*$ in $A$, $X$, so that both $f$, $\alpha$ are base-point preserving.

 Firstly, we define a map 
 \[
 \varphi:\Loop_f X \times_{X} \Loop_f X \to \Loop X, 
 \]
 where 
 \[
 \Loop_f X \times_{X}\Loop_f X=\{ (\lambda_1, \lambda_2) \in \Loop_f X \times \Loop_f X\;|\; \lambda_1(0)=\lambda_2(0) \}.
 \]
 The map $\varphi$ is defined by
\[
 \varphi(\lambda_1, \lambda_2)(t)=
 \begin{cases}
\lambda_{1}(2t) & \text{for  $\displaystyle 0\leq t \leq \frac{1}{2}$,} \\
\lambda_{2}(2-2t) & \text{for  $\displaystyle \frac{1}{2} \leq t \leq 1$.}
\end{cases}
 \]
 Since $\lambda_{1}(1)=f(\lambda_1(0))$, $\lambda_2(1)=f(\lambda_2(0))$ and $\lambda_{1}(0)=\lambda_2(0)$, this map is well-defined.
 
 Next, we define a map from $P_{\alpha}$ to $\Loop_f X$, say $\psi:P_{\alpha} \to \Loop_f X$, by
 \[
 \psi((a, \lambda))(t)=
 \begin{cases}
\lambda(2t) & \text{for  $\displaystyle 0\leq t \leq \frac{1}{2}$,} \\
f(\lambda(2-2t)) & \text{for  $\displaystyle \frac{1}{2} \leq t \leq 1$.}
 \end{cases}
 \]
Since $\lambda(1)=f(\lambda(1))$, this map is also well-defined.

Now, we consider the following diagram:
\[
\begin{diagram}
\node{\Loop_f X}   \node{\Loop_f X \times_{X} \Loop_f X} \arrow{w,t}{p_1} \arrow{e,t}{\varphi} \node{\Loop X} \\
\node[2]{\Loop_f X \times_X \MappingTrack{\alpha}.} \arrow{n,r}{1\times \psi}
\end{diagram}
\]
where
\[
\Loop_f X \times_{X} P_{\alpha}=\{ (\lambda_1, (a, \lambda_2) )\in \Loop_f X \times P_{\alpha}\;|\; \lambda_1(0)=\lambda_2(0),\;  \alpha(a)=\lambda_2(1) \},
\]
$p_1$ is the projection onto the first factor and $\pi_0\circ p_1=\pi_0$, $\pi_0\circ \varphi=\pi_0$, $\pi_{0}\circ (1\times \psi)=\pi_0$.
Let us denote by $E_r(Y)$ the Leray-Serre spectral sequence associated with a  fibration $\xi:Y \to X$. Then we have the following diagram of spectral sequences:
\[
\begin{diagram}
\node{E_r(\Loop_f X)} \arrow{e,t}{p_1^*}  \node{E_r(\Loop_f X \times_{X} \Loop_f X)} \arrow{s,r}{1\times \psi^*} \node{E_r(\Loop X)} \arrow{w,t}{\varphi^*}\\
\node[2]{E_r(\Loop_f X \times_X P_{\alpha}).}
\end{diagram}
\]
By abuse of notation, we denote by $\psi:\thinspace \HomotopyFibre{\alpha} \to \Omega X$ the restriction of $\psi:\thinspace P_{\alpha} \to \Loop_f X$ to fibres.
Let us consider a sufficient condition for the induced homomorphism 
\[
\psi^*:\ReducedCohomology{\Omega X}  \to \ReducedCohomology{\HomotopyFibre{\alpha}}
\]
to be zero.
Again, by abuse of notation, we denote by 
$\varphi:\Omega X \times \Omega X\to \Omega X$ 
the restriction of $\varphi:\Loop_f X \times_{X} \Loop_f X\to \Loop X$ to fibres.

%:  Lemma 3.1 [same]

\begin{lemma}\label{Lemma:same}
If the induced homomorphism
\[
(1 \times (1\times f))^*:\Cohomology{\HomotopyFibre{\Delta\circ \alpha}} \to \Cohomology{\HomotopyFibre{\Delta\circ \alpha}} 
\]
is the identity homomorphism, then the induced homomorphism
\[
\psi^*:\ReducedCohomology{\Omega X} \to \ReducedCohomology{\HomotopyFibre{\alpha}}
\]
is zero.
\end{lemma}

\begin{proof}
The map $\psi:\HomotopyFibre{\alpha} \to \Omega X$ factors through
\[
\HomotopyFibre{\alpha} \stackrel{1\times \Delta}{\longrightarrow} \HomotopyFibre{\Delta\circ \alpha} \stackrel{1\times (1\times f)}{\longrightarrow} \HomotopyFibre{\Delta\circ \alpha} \stackrel{\varphi}{\longrightarrow} \Omega X.
\]
It is clear that the composition $\varphi\circ \Delta$ is null homotopic since an obvious null homotopy $h_s$ is given by
\[
h_s((a,\lambda))(t)=
\begin{cases}
\lambda(2st)& \text{for  $\displaystyle 0\leq t \leq \frac{1}{2}$,} \ \\
\lambda(2s-2st)& \text{for  $\displaystyle \frac{1}{2} \leq t \leq 1$.}
\end{cases}
\]
Therefore, we have 
\begin{align*}
(1\times \Delta)^*( (1\times (1\times f))^* (\varphi^*(x)))&=(1\times  \Delta)^*(\varphi^*(x)) \\
&=0.
\end{align*}
for $x \in \ReducedCohomology{\Omega X}$.
\end{proof}

%:  Lemma 3.2 [mono]

We also need the following lemmas in the proof of Theorem~\ref{Theorem:main}.

\begin{lemma} \label{Lemma:mono}
Suppose that $X$ is simply connected, that  $H^i(\HomotopyFibre{\alpha})=0$ for $i>k$ and 
that there exists a sequence of maps  
\[
\begin{diagram}
\node{A_0} \arrow{e,t}{i_0} \node{A_1} \arrow{e,t}{i_1} \node{A_2} \arrow{e} \node{\dots} \arrow{e} \node{A_{k}} 
\arrow{e,t}{i_k} \node{A} \arrow{e,t}{\alpha}
\node{X}
\end{diagram}
\]
such that the induced homomorphism $
\ReducedCohomology{\HomotopyFibre{\alpha_{j}}} \to \ReducedCohomology{\HomotopyFibre{\alpha_{j-1}}}
$
is zero for $j=1, 2, \dots, k$.
Then the projection on the first factor $p_1:Y \times_{X} \MappingTrack{\alpha} \to Y$
 induces a monomorphism $p_1^{*}:\thinspace E_r(Y) \to E_r(Y \times_{X} \MappingTrack{\alpha})$ of Leray-Serre spectral sequences
 for arbitrary fibration $\xi: Y \to X$.
\end{lemma}

%:  Lemma 3.3 [ss]

We need the following lemma to compare the spectral sequences.

\begin{lemma} \label{Lemma:ss}
Let 
\[
E_r' \stackrel{\rho_r'}{\longrightarrow} E_r \stackrel{\rho''_r}{\longleftarrow} E_r''
\]
be homomorphisms of spectral sequences. 
Suppose that 
\begin{itemize}
\item[{\rm (1)}] $\mathrm{Im}\, \rho_2'=\mathrm{Im}\, \rho_2''$,
\item[{\rm (2)}] $\rho_r'$ is a monomorphism for $r\geq 2$.
\end{itemize}
Then, there exists a unique homomorphism of spectral sequences $$\{ \tau_{r}:E_r'' \to E_r' \;|\; r \geq 2 \}$$
such that $\rho'_r \circ \tau_r=\rho_r''$ for $r\geq 2$. In particular, if $\rho''_2$ is also a monomorphism, then $\rho''_r$ is a monomorphism and $\tau_r$ is an isomorphism for $r\geq 2$.
\end{lemma}

\begin{proof}
We define $\tau_2(x'')$ by 
\[
\rho'_2 (\tau_{2}(x''))=\rho_2'' (x'').
\]
Since $\mathrm{Im}\thinspace \rho_2'=\mathrm{Im}\thinspace \rho_2''$ and $\rho_2'$ is a monomorphism, 
it is well-defined and we have 
\[
\rho_2' \circ \tau_2=\rho_2''
\]
at the $E_2$-level.
Suppose that we have 
\[
\rho_r' \circ  \tau_r=\rho_r''.
\]
Then, we want to show that 
\[
d_r' (\tau_r(x''))=\tau_r(d_r''(x'')).
\]
Since $\rho_r'$ is  a monomorphism, it suffices to show that 
\[
\rho'_r (d_r'( \tau_r(x'')))=\rho'_r( \tau_r (d_r''(x''))).
\]
It is easily verified as follows:
\begin{align*} 
\rho'_r (d_r'( \tau_r (x''))) & = d_r (\rho'_r (\tau_r(x'') ))\\
&=d_r (\rho''_r (x'')) \\
&= \rho''_r (d_r'' (x'')) \\
& =\rho_r ' (\tau_r (d_r'' (x''))).
\end{align*}
Then, $\tau_r$ induces a homomorphism 
\[
\tau_{r+1}:\thinspace E''_{r+1} \to E'_{r+1}
\]
such that 
\[
\rho'_{r+1} \circ \tau_{r+1}=\rho''_{r+1}.
\]
Continue this process, we have a homomorphism of spectral sequence 
\[
\tau_r:\thinspace E_{r}'' \to E_r'
\]
for $r\geq 2$. It is clear that if $\rho_2''$ is a monomorphism, then $\tau_2$ is an isomorphism.
It is also clear from the construction  that $\rho_r''$ is a monomorphism for $r\geq 2$ and $\tau_r$ is an isomorphism for $r\geq 2$.
\end{proof}

%:  Proof of Theorem:main

Now, we prove Theorem~\ref{Theorem:main} assuming Lemma~\ref{Lemma:mono}.

\begin{proof}[Proof of Theorem~\ref{Theorem:main}]
Let $k=\dim G$.
Let $q_j=p^{e_1^{j}}$ for $j=0, \dots, k$.
Let  $q=p^{ab}={q_k}^{a e_2}$ ($a\geq 1$).
Let $X=BG^{\wedge}$, $A=BG(\mathbb{F}_{q_k})$, $\alpha=D_{q_k}$, $f=\phi^{q}$,
$A_j=BG(\mathbb{F}_{q_j})$ and $\alpha_i=D_{q_j}$ for $j=0,1,\dots, k$.
In order to prove Theorem~\ref{Theorem:main},  we consider the Leray-Serre spectral sequence $E_r(\Loop_f X)$, $E_r(\Loop X)$ and establish an isomorphism of spectral sequences $\tau:\thinspace E_r(\Loop X) \to E_r(\Loop_f X)$. 

By Lemma~\ref{Lemma:zero}, we have that the induced homomorphism
\[
\ReducedCohomology{\HomotopyFibre{\alpha_j}}\to \ReducedCohomology{\HomotopyFibre{\alpha_{j-1}}}
\]
is zero for $j=1, \dots, k$.
By Lemma~\ref{Lemma:mono}, we have
a monomorphism
\[
(1\times \psi)^*\circ p_1^*:E_r(\Loop_f X) \longrightarrow E_r(\Loop_f X \times_{X} P_{\alpha}).
\]
The fibres of fibrations $\pi_0:\Loop_f X \to X$, $\pi_0:\Loop X \to X$, $\pi_0:\Loop_f X \times_{X} P_{\alpha} \to X$ are $\Omega X$, $\Omega X$, $\Omega X \times \HomotopyFibre{\alpha}$, respectively.
Identifying the $E_2$-terms $E_2(\Loop_f X)$, $E_2(\Loop X)$, $E_2(\Loop_f X \times_{X} P_{\alpha})$ of Leray-Serre spectral sequences with $\Cohomology{X} \otimes \Cohomology{\Omega X}$, $\Cohomology{X} \otimes \Cohomology{\Omega X}$, $\Cohomology{X}\otimes \Cohomology{\Omega X} \otimes \Cohomology{\HomotopyFibre{\alpha}}$, respectively, we have
\begin{align*}
(1 \times \psi)^*(p_1^*(x \otimes y ))&= x \otimes y \otimes 1\quad \text{and}\\
(1\times \psi)^*(\varphi^*(x\otimes y))&= \sum x \otimes y' \otimes \psi^*( \chi(y'')),
\end{align*}
where $\varphi^{*}(y)=\sum y' \otimes \chi(y'')$ and  $\chi:\Cohomology{\Omega X} \to \Cohomology{\Omega X}$ is the canonical anti-automorphism of connected Hopf algebra over $\mathbb{Z}/\ell$.
By the definition of $f$, we have 
$f=(\phi^{q_k})^{ae_2}$ and $ae_2$ is divisible by $e_2$. So, by Lemma~\ref{Lemma:identity}, 
the induced homomorphism
$(1\times (1\times f))^*$ is the identity homomorphism.
By Lemma~\ref{Lemma:same}, 
we obtain 
$$\mathrm{Im}\, (1 \times \psi)^* \circ p_1^*=\mathrm{Im}\, (1 \times \psi)^*\circ \varphi^*$$
in the $E_2$-term $E_2(\Loop_f X \times_{X} P_{\alpha})=H^{*}(X) \otimes \Cohomology{\Omega X} \otimes \Cohomology{\HomotopyFibre{\alpha}}$.
Therefore, using Lemma~\ref{Lemma:ss}, we obtain an isomorphism between Leray-Serre spectral sequences $E_r(\Loop_f X)$ and $E_r(\Loop X)$.
\end{proof}

%%% section  4 %%%%%%%%%%%%%%%%%%%%%%%%%%

\section{Proof of Lemma~\ref{Lemma:mono}}\label{sec4}

%: Lemma 4.1 [monoA]

In order to prove Lemma~\ref{Lemma:mono}, 
we need to recall the internal structure of Leray-Serre spectral sequence. 
Let $\eta:\thinspace Z\to Y$, $\xi:\thinspace Y \to X$ be fibrations.
Without loss of generality, we may assume that $X$ is a $CW$ complex.
Suppose $X$ is a CW complex and denote its $n$-skelton by $X^{(n)}$. We denote $\xi^{-1}(X^{(n)}) \subset Y$, 
$\eta^{-1}(\xi^{-1}(X^{(n)}))$  by $F_nY$, $F_n Z$, respectively and let $F_nY=F_{n}Z=\emptyset$ for $n<0$.
For the sake of notational simplicity, we let 
\begin{align*}
M_{m,n}(Y)&=\Cohomology{F_{m}Y, F_nY}, \\
 M_{m,n}(Z)&=\Cohomology{F_{m}Z, F_{n}Z},
\end{align*}
respectively.
We denote by $E_r(Y)$, $E_r(Z)$ the Leray-Serre spectral sequences associated with fibrations $\xi$, $\xi\circ \eta$, respectively.

\begin{lemma} \label{Lemma:monoA}
Suppose that for $m\geq n \geq 0$, the induced homomorphism 
\[
\eta^*:\thinspace M_{m,n}(Y)\longrightarrow M_{m,n}(Z)
\]
is a monomorphism. Then, the induced homomorphism
\[
\eta^{*}:\thinspace  E_r(Y) \to E_r(Z)
\]
is also a monomorphism for $r\geq 2$.
\end{lemma}

\begin{proof}
Let us consider  the following diagram:
\[
\begin{diagram}
\node{\Cohomology{Y, F_{s-r}Y} } \node[2]{M_{s+r,s}(Y)} \\
\node{\Cohomology{Y,F_{s-1}Y} } \arrow{n,l}{j_r} \arrow{e,t}{i_1} \node{M_{s,s-1}(Y) } \arrow{ne,t}{\delta_{r}'} \arrow{e,t}{\delta_1} \node{\Cohomology{Y,F_sY}} \arrow{n,r}{i_r} \\
\node{M_{s-1,s-r}(Y) } \arrow{n,l}{\delta_r} \arrow{ne,b}{\delta_{r}''} \node[2]{\Cohomology{Y,F_{s+r}Y}. } \arrow{n,r}{j_r}
\end{diagram}
\]
Let 
$
Z_r(Y) = \mathrm{Ker}\thinspace \delta_{r}' $ and
$B_r(Y) = \mathrm{Im}\thinspace \delta_{r}''. $
Then, there holds $
E_r(Y) = Z_r(Y)/B_r(Y)
$
for $r\geq 2$.
See standard text books, for instance, McCleary's book \cite{mccleary}, for detail. 
We consider the same diagram and $Z_r(Z)$, $B_r(Z)$ for $Z$ in the same manner. Then, we have $E_r(Z)=Z_r(Z)/B_r(Z)$.
 Thus,
in order to prove the injectivity of the induced homomorphism $\eta^*:E_r(Y) \to E_r(Z)$, 
it suffices to show that 
\[
{\eta}^{*}(Z_{r}(Y)) \cap B_r(Z )= {\eta}^*(B_r(Y)), 
\]
that is, if ${\eta}^{*}(x) \in \mathrm{Im}\thinspace \delta_{r}''$ in 
$M_{s,s-1}(Z)$, then
$x\in \mathrm{Im}\thinspace \delta_{r}''$ in $M_{s, s-1}(Y)$.
We consider the following diagram:
\[
\begin{diagram}
\node{ M_{s-1,s-r}(Z)} 
\arrow{e,t}{\delta_r''} 
\node{M_{s,s-1} (Z) } 
\arrow{e,t}{j_{r}''} 
\node{ M_{s,s-r}(Z) } 
\\
\node{M_{s-1,s-r}(Y) } \arrow{n,r}{\eta^*} \arrow{e,t}{\delta_{r}''} 
\node{M_{s,s-1} (Y) } \arrow{n,r}{\eta^*} \arrow{e,t}{j_{r}''} \node{M_{s,s-r}(Y),}\arrow{n,r}{\eta^*}
\end{diagram}
\]
where horizontal sequences are cohomology long exact sequences associated with triples $(F_{s} Z, F_{s-1}Z, F_{s-r}Z)$, $(F_sY, F_{s-1}Y, F_{s-r}Y)$.
Suppose $\eta^*(x) \in \mathrm{Im}\thinspace \delta_{r}''$, then $j''_{r}(\eta^{*}(x))=0$. So, we have that $\eta^* (j_{r}''(x))=0$. Since $\eta^*$ is a monomorphism, we have $j''_{r}(x)=0$. Hence, we have $x \in \mathrm{Im}\thinspace \delta_r''$. This completes the proof.
\end{proof}

%: Proof of Lemma 3.2 [mono] / Proposition 4.2 [monoB]

Now, we complete the proof of Lemma~\ref{Lemma:mono}.
Recall that $P_{\alpha_j} \to X$ is the fibration with the fibre $\HomotopyFibre{\alpha_{j}}$.
Let $Y_{j}=Y \times_{X} P_{\alpha_{j}}$ for $j=0,\dots, k$.  There is a sequence of fibrations and fibre maps over $X$,
\[
\begin{diagram}
\node{Y_0} \arrow{e,t}{i'_0} \node{Y_1}  \arrow{e,t}{i'_1} \node{\dots} \arrow{e,t}{i'_{k-1}} \node{Y_k}\arrow{e,t}{i'_k}  \node{Y\times_{X} P_{\alpha}} \arrow{e,t}{p_1} \node{Y,}\end{diagram}
\]
where $i'_j=1\times i_j \times 1$ for $j=0, \dots, k$.
We denote the projection onto the first factor from $Y_j$ to $Y$ by $\eta_j$.
The following diagram is a fibre square and the fibre of $\eta_{j}$ is also $\HomotopyFibre{\alpha_{j}}$.
\[
\begin{diagram}
\node{Y_j} \arrow{s,l}{\eta_{j}} \arrow{e,t}{p_2} \node{P_{\alpha_j}} \arrow{s,r}{\pi_0} \\
\node{Y} \arrow{e,t}{\xi} \node{X.}
\end{diagram}
\]

\begin{proposition}\label{Proposition:monoB}
Let $Y''\subset Y' \subset Y$ be subspaces of $Y$ and let $Y''_{j}=\eta_{j}^{-1}(Y'')$, $Y'_j=\eta_{j}^{-1}(Y')$ be subspaces of $Y_j$.
Suppose that $X$ is simply connected and that $H^{i}(\HomotopyFibre{\alpha})=\{0\}$ for $i >k$.
The induced homomrphism 
\[
\Cohomology{Y',Y''} \to \Cohomology{Y'\times_{X} \MappingTrack{\alpha}, Y''\times_{X} \MappingTrack{\alpha}}
\]
is a monomorphism.
\end{proposition}

\begin{proof}
We have  relative fibrations
\[
\HomotopyFibre{\alpha} \to (Y'\times_{X} \MappingTrack{\alpha}, Y''\times_{X} \MappingTrack{\alpha}) \to (Y', Y'') \quad \text{and}\quad \HomotopyFibre{\alpha_{j}} \to ( Y'_{j}, Y''_{j}) \to (Y', Y'')
\]
for $j=0,\dots,  k$.

There exist  associated Leray-Serre spectral sequences
\[
\quad E_r(Y'\times_{X} \MappingTrack{\alpha}, Y''\times_{X} \MappingTrack{\alpha}) \quad \text{and} \quad E_r(Y_{j}', Y''_{j}), 
\]
converging to 
\[
\gr \Cohomology{Y'\times_{X} \MappingTrack{\alpha}, Y''\times_{X} \MappingTrack{\alpha}}
\quad \text{and} \quad   \gr\Cohomology{Y_{j}', Y_{j}'' },
\]
for $j=0, \dots, k$, respectively.
The fundamental group $\pi_{1}(Y)$ acts on $\Homology{\HomotopyFibre{\alpha_j}}$. This action factors through $\pi_{1}(X)=\{0\}$. Therefore, the action of $\pi_{1}(Y)$ on $\Homology{\HomotopyFibre{\alpha_j}}$ is trivial. Hence, its action on $\Cohomology{\HomotopyFibre{\alpha_j}}$ is also trivial. Thus, the $E_2$-term of the Leray-Serre spectral sequence associated with 
the relative fibration 
\[
\HomotopyFibre{\alpha_{j}} \to (Y'_{j}, Y''_{j}) \to (Y', Y''),
\]
is 
\[
E_2(Y_{j}',Y_{j}'')=\Cohomology{Y',Y''} \otimes \Cohomology{\HomotopyFibre{\alpha_{j}}}.
\]
Suppose that $d_{r_{k}}(y)=z$ in $E_{r_k}^{s,0}(Y'\times_{X} \MappingTrack{\alpha},Y''\times_{X} \MappingTrack{\alpha})$ for some $y \in E_{r_k}^{s-r_k, r_{k}-1}(Y'\times_{X} \MappingTrack{\alpha},Y''\times_{X} \MappingTrack{\alpha})$. 
Let $z_k={i'_{k}}^*(z)$ and $y_k={i'_{k}}^*(y)$.
Since $r_k-1>0$, we have ${i'_{k-1}}^*(y_k)=0$. Therefore,  ${i'_{k-1}}^{*}(z_{k})$ in $E_{r_k}^{s,0}(Y'_{k-1}, Y''_{k-1})$ is also zero.
So, for some $r_{k-1}<r_{k}$, ${i'_{k-1}}^{*}(z_k)$ in $E_{r_{k-1}}^{s,0}(Y'_{k-1}, Y''_{k-1})$ must be hit, that is,  there exists $y_{k-1}$ in $E_{r_{k-1}}^{s-r_{k-1},r_{k-1}-1}(Y_{k-1}', Y_{k-1}'')$  such that $d_{r_{k-1}}(y_{k-1})=i_{k-1}^*(z_{k})$ in $E_{r_{k-1}}^{s,0}(Y'_{k-1}, Y''_{k-1})$. 
Continuing this precess, we have a sequence of integers
\[
2 \leq r_0<r_1<\dots < r_k.
\]
Hence, we have $r_k\geq k+2$. However, $d_r=0$ for $r \geq k+2$ in $$E_{r}(Y'\times_{X} \MappingTrack{\alpha},Y''\times_{X} \MappingTrack{\alpha}).$$
It is a contradiction. So, each element in $H^{s}(Y',Y'')=E_{r}^{s,0}(Y' \times_{X} \MappingTrack{\alpha}, Y''\times_{X}\MappingTrack{\alpha})$ is not hit in $E_r(Y'\times_{X} \MappingTrack{\alpha},Y''\times_{X}\MappingTrack{\alpha})$ for $r\geq 2$. In other words, it is a permanent cocycle.
Therefore, the induced homomorphism 
\[
\Cohomology{Y',Y''} \to \Cohomology{Y'\times_{X}\MappingTrack{\alpha}, Y''\times_{X}\MappingTrack{\alpha}}
\]
is a monomorphism.
\end{proof}

By Proposition~\ref{Proposition:monoB}, we have that the induced homomorphism $$\eta^*:M_{m,n}(Y) \to M_{m,n}(Z)$$ is a monomorphism
for $Z=Y \times_{X} P_{\alpha}$, $\eta=p_1:Y \times_{X} P_{\alpha} \to Y$.
Thus, Lemma~\ref{Lemma:monoA} completes the proof of Lemma~\ref{Lemma:mono}.

%:References

\end{document}